\documentclass[12pt,draft,reqno]{amsart}
\usepackage{amsmath}
\usepackage{amssymb}
\usepackage{amsthm}
\usepackage{mathrsfs}
\usepackage{bm}
\numberwithin{equation}{section}
\newtheorem{theorem}{Theorem}[section]

\newtheorem{corollary}[theorem]{Corollary}

\newtheorem{thm}{Theorem}

\theoremstyle{definition}
\newtheorem*{remark}{Remark}

\allowdisplaybreaks
\begin{document}
\title[Dirichlet series with Riemann's FE and real zeros of Dirichlet {\textit{L}}-functions]{Dirichlet series with periodic coefficients, Riemann's functional equation and real zeros of Dirichlet {\textit{L}}-functions}
\author[T.~Nakamura]{Takashi Nakamura}
\address[T.~Nakamura]{Department of Liberal Arts, Faculty of Science and Technology, Tokyo University of Science, 2641 Yamazaki, Noda-shi, Chiba-ken, 278-8510, Japan}
\email{nakamuratakashi@rs.tus.ac.jp}
\urladdr{https://sites.google.com/site/takashinakamurazeta/}
\maketitle

\begin{abstract}
In this paper, we give Dirichlet series with periodic coefficients that have Riemann's functional equation and real zeros of Dirichlet $L$-functions. The details are as follows. 
Let $L(s,\chi)$ be the Dirichlet $L$-function and $G(\chi)$ be the Gauss sum associate with a primitive Dirichlet character $\chi$ (${\rm{mod}} \,\, q$). Put $f (s,\chi) :=  q^s L(s,\chi) + i^{-\kappa (\chi)} G(\chi) L(s,\overline{\chi})$, where $\overline{\chi}$ is the complex conjugate of $\chi$ and $\kappa (\chi) :=(1-\chi (-1))/2$. Then we prove that $f (s,\chi)$ satisfies Riemann's functional equation appearing in Hamburger's theorem if $\chi$ is even. In addition, we show that $f (\sigma,\chi) \ne 0$ all $\sigma \ge 1$. Moreover, we prove that $f(\sigma,\chi) \ne 0$ for all $1/2 \le \sigma < 1$ if and only if $L(\sigma,\chi) \ne 0$ for all $1/2 \le \sigma < 1$. When $\chi$ is real, all zeros of $f(s,\chi)$ with $\Re (s) >0$ are on the line $\sigma =1/2$ if and only if GRH for $L(s,\chi)$ is true. However, $f (s,\chi)$ has infinitely many zeros off the critical line $\sigma =1/2$ if $\chi$ is non-real.
\end{abstract}

\section{Introduction}

\subsection{The Riemann zeta function and the Dirichlet {\textit{L}}-function}
Let $s$ be a complex variable expressed as $s = \sigma +it$, where $\sigma,t \in {\mathbb{R}}$. Then, the Riemann zeta function is defined by the series
\[
\zeta (s) := \sum_{n=1}^\infty \frac{1}{n^s}, \qquad \sigma >1.
\]
It is widely known that $\zeta (s)$ is continued meromorphically and has a simple pole at $s=1$ with residue $1$. The Dirichlet $L$-function is defined by the series
\[
L(s,\chi) := \sum_{n=1}^\infty \frac{\chi (n)}{n^s}, \qquad \sigma >1,
\]
where $\chi (n)$ is a Dirichlet character (${\rm{mod}} \,\, q$). The Dirichlet $L$-function $L(s,\chi)$ can be analytically continued to the whole complex plane to a holomorphic function if $B_0(\chi) : = \sum_{r=0}^{q-1} \chi (r) /q =0$, otherwise to a meromorphic function with a simple pole, at $s=1$, with residue $B_0(\chi)$ (see for example \cite[Corollary 10.2.3]{Cohen}). For simplicity, put
\[
\Gamma_{\! \pi} (s) : = \frac{\Gamma (s)}{(2\pi )^s}, \qquad
\Gamma_{\!\! \rm{cos}} (s) : = 2 \Gamma_{\! \pi} (s) \cos \Bigl( \frac{\pi s}{2} \Bigr),  \qquad 
\Gamma_{\!\! \rm{sin}} (s) : = 2 \Gamma_{\! \pi} (s) \sin \Bigl( \frac{\pi s}{2} \Bigr) .
\]
Then, the Riemann zeta function $\zeta (s)$ satisfies Riemann's functional equation 
\begin{equation}\label{eq:RZfe}
\zeta(1-s) = \Gamma_{\!\! \rm{cos}} (s) \zeta(s)
\end{equation}
(see for example \cite[(2.1.8)]{Tit}). When $\chi$ is a primitive character (${\rm{mod}} \,\, q$), the Dirichlet $L$-function $L(s,\chi)$ satisfies the functional equation
\begin{equation}\label{eq:DLfe}
L(1-s, \chi) = q^{s-1} \Gamma_{\! \pi} (s) \bigl( e^{-\pi is/2} + \chi(-1) e^{\pi is/2} \bigr) G(\chi) L(s,\overline{\chi}),
\end{equation}
where $\overline{\chi}$ is the complex conjugate of $\chi$ and the Gauss sum $G(\chi)$ is defined as $G(\chi) := \sum_{r=1}^q \chi (r) e^{2\pi i r/q}$ (see for example \cite[Theorem 12.11]{Apo}).

Next, we review some of the standard conjectures and facts on zeros of $\zeta (s)$ and $L(s,\chi)$. All real zeros of $\zeta(s)$, so-called trivial zeros, are simple and only at the negative even integers. The famous Riemann hypothesis is a conjecture that the Riemann zeta function $\zeta(s)$ has its non-real zeros only on the critical line $\sigma =1/2$. When $\chi$ is an even primitive character, then the only zeros of $L(s,\chi)$ with $\Re(s) < 0$ are at the negative even integers. If $\chi$ is an odd primitive character, then the only zeros of $L(s,\chi)$ with $\Re(s) < 0$ are at the negative odd integers. Note that $\zeta (s)$ and $L(s, \chi)$ do not vanish when $\sigma >1$ by their Euler products. The generalized Riemann hypothesis asserts that, for every Dirichlet character $\chi$ and every complex number $s$ with $\Re (s) > 0$ and $L(s,\chi) = 0$, then the real part of $s$ is 1/2. Note that the existence of a real zero at $s=1/2$ is not prohibited by the GRH. 

\subsection{Dirichlet series with periodic coefficients}
For $\sigma >1$, define two Dirichlet $L$-functions  (${\rm{mod}} \,\, 5$) by
\[
L_1 (s) := \frac{1}{1^s} + \frac{i}{2^s} - \frac{i}{3^s} - \frac{1}{4^s} + \frac{1}{6^s} + \cdots , \qquad 
L_2 (s) := \frac{1}{1^s} - \frac{i}{2^s} + \frac{i}{3^s} - \frac{1}{4^s} + \frac{1}{6^s} + \cdots .
\]
And we define real numbers $0 < \theta < \pi /2$ and $\xi>0$ by
\[
\tan \theta := \xi, \qquad \xi := \bigl( 5^{1/2} -1 \bigr)^{-1} \bigl( (10 - 2 \cdot 5^{1/2})^{1/2} - 2 \bigr) .
\]
Then, the Davenport-Heilbronn function (see \cite[Chapter 10.25]{Tit}) is defined by
\begin{equation*}\label{eq:defDH}
f(s) := \frac{\sec \theta}{2} \Bigl( e^{-i \theta} L_1 (s) + e^{i \theta} L_2 (s) \Bigr)  .
\end{equation*}

We have the following for the Davenport-Heilbronn function $f(s)$ (see \cite[Section 1]{BS}, \cite[Section 5]{BG} or \cite[Chapter 10.25]{Tit}). It should be emphasised that the gamma factor of the functional equation (\ref{eq:oDHfe}) below is similar to that of (\ref{eq:RZfe}) but definitely different from that of (\ref{eq:RZfe}). More precisely, the Riemann zeta function $\zeta (s)$ has trivial zeros at the negative {\sc{even}} integers but the Davenport-Heilbronn function $f(s)$ has trivial zeros at the negative {\sc{odd}} integers.

\begin{thm}
The function $f(s)$ satisfies the functional equation
\begin{equation}\label{eq:oDHfe}
f(1-s) = 5^{s-1/2} \Gamma_{\!\! \rm{sin}} (s) f(s) .
\end{equation}
Furthermore, the function $f(s)$ has an infinity of zeros in the half-plane $\sigma >1$. 
\end{thm}

Recently, Vaughan \cite{Vau} studied the following functions. Let 
\[
\lambda (\chi) := \frac{G(\chi)}{i^{\kappa(\chi)}\sqrt{q}},\qquad
\kappa(\chi) := 
\begin{cases}
1 & \chi \mbox{ is odd},\\
0 & \chi \mbox{ is even}.
\end{cases} \qquad
\]
Then we define the functions $V_+(s, \chi)$ and $V_-(s, \chi)$ by
\[
V_+(s, \chi) := \frac{L(s,\chi) + \lambda(\chi) L(s,\overline{\chi})}{1 + \lambda(\chi)} \quad \mbox{and} \quad
V_-(s, \chi) := \frac{L(s,\chi) - \lambda(\chi) L(s,\overline{\chi})}{1 - \lambda(\chi)}
\]
when $\lambda(\chi) \ne -1$ and $\lambda(\chi) \ne 1$, respectively. 
Then, he showed the following (see \cite[Theorems 2.1, 2.2, 3.1 and 3.2]{Vau}). 
\begin{thm}
We have the functional equations
\begin{equation*}
\begin{split}
V_+ (1-s, \chi)  & =   q^{s-1/2} \Gamma_{\!\! \rm{cos}} (s-\kappa(\chi)) V_+(s, \chi), \\ 
V_- (1-s, \chi) & = - q^{s-1/2} \Gamma_{\!\! \rm{cos}} (s-\kappa(\chi)) V_-(s, \chi).
\end{split}
\end{equation*}
Moreover, the function $V_+(s, \chi)$  has $\asymp T$ zeros in the region $\{ s : \sigma > 1, \,\, \lvert t \rvert < T \}$. The same statement also holds for $V_-(s, \chi)$. 
\end{thm}

The existence of complex zeros of $f(s)$ and $V_\pm (s,\chi)$ are proved by the theorem below. Let $N_F (\sigma_1,\sigma_2, T)$ be the number of zeros of a function $F(s)$ in the rectangle $\sigma_1 < \Re (s) < \sigma_2$, $\lvert \Im (s) \rvert \le T$, counted with their multiplicities. Then we have the  following which is proved by Saias and Weingartner \cite[Theorem]{SW}. 
\begin{thm}\label{thm:SW1}
Let $\{ a (n) \}$ be a periodic sequence of complex numbers such that $F(s) := \sum_{n=1}^\infty a(n) n^{-s}$ is not of the form $P(s) L(s,\chi)$, where $P(s)$ is a Dirichlet polynomial and $L(s,\chi)$ is a Dirichlet $L$-function. Then there exists a positive number $\eta$ such that, for all real numbers $\sigma_1$ and $\sigma_2$ with $1/2 < \sigma_1 < \sigma_2 < 1+\eta$, there exist positive numbers $c_1$, $c_2$, and $T_0$ such that for all $T \ge T_0$ we have
\[
c_1 T \le N_F (\sigma_1,\sigma_2, T) \le c_2 T.
\]
\end{thm}

\section{The main theorem and remarks}

\subsection{Main theorem}

As an analogue or improvement of the Davenport-Heilbronn function $f(s)$ or Vaughan functions $V_\pm (s,\chi)$, we define the function
\begin{equation*}\label{eq:defgDH}
f (s,\chi) :=  q^s L(s,\chi) + i^{-\kappa (\chi)} G(\chi) L(s,\overline{\chi}) ,
\end{equation*}
where $\chi$ is a primitive Dirichlet characters (${\rm{mod}} \,\, q$) and $\kappa (\chi) :=(1-\chi (-1))/2$ . Note that all the functions $f(s)$,  $V_\pm(s,\chi)$ and $f(s,\chi)$ are expressed as a sum of two Dirichlet $L$-functions. In the present paper, we show the following which implies that $f(s,\chi)$ have Riemann's functional equation (if $\chi$ is even) and real zeros of Dirichlet {\textit{L}}-functions. 

\begin{theorem}\label{th:gDH1} We have the four statements below;

\item[$\,\, {\bf{(i).}}$] The function $f(s,\chi)$ satisfies Riemann's functional equation 
\begin{equation}\label{eq:gDH1fe}
f(1-s, \chi) = \Gamma_{\!\! \rm{cos}} (s) f(s, \chi)
\end{equation}
if $\chi (-1) =1$. When $\chi (-1) =-1$, the function $f(s,\chi)$ fulfills
\begin{equation}\label{eq:gDH2fe}
f(1-s, \chi) = \Gamma_{\!\! \rm{sin}} (s) f(s, \chi) .
\end{equation}

\item[$\, {\bf{(ii).}}$]
One has $f(\sigma,\chi) \ne 0$ for all $\sigma \ge 1$ and primitive Dirichlet character $\chi$. Moreover, we have $f(\sigma,\chi) \ne 0$ for all $1/2 \le \sigma < 1$ if and only if $L(\sigma,\chi) \ne 0$ for all $1/2 \le \sigma < 1$. 

\item[${\bf{(iii).}}$]
Fix a real Dirichlet character $\chi$. Then all complex zeros of $f(s,\chi)$ with $\Re (s) >0$ are on the vertical line $\sigma =1/2$ if and only if GRH for $L(s,\chi)$ is true.

\item[${\bf{(iv).}}$]
Fix a non-real Dirichlet character $\chi$. Then, there exists a positive number $\eta$ such that, for all real numbers $\sigma_1$ and $\sigma_2$ with $1/2 < \sigma_1 < \sigma_2 < 1+\eta$, there exist positive numbers $C_1$, $C_2$, and $T_0$ such that for all $T \ge T_0$ we have
\[
C_1 T \le N_{f(s, \chi)} (\sigma_1,\sigma_2, T) \le C_2 T,
\]
where $N_{f(s, \chi)} (\sigma_1,\sigma_2, T)$ is the number of zeros of the function $f(s, \chi)$ in the rectangle $\sigma_1 < \Re (s) < \sigma_2$, $\lvert \Im (s) \rvert \le T$, counted with their multiplicities.
\end{theorem}

We remark that the gamma factor of the functional equation (\ref{eq:gDH1fe}) completely coincides with that of Riemann's functional equation (\ref{eq:RZfe}) or (H3) given below. Moreover, the statement (ii) of Theorem \ref{th:gDH1} implies that the real zeros of $f(s,\chi)$ completely coincide with those of $L(s,\chi)$ (there are no theorem on real zeros of $f(s)$ or $V_\pm (s, \chi)$ in \cite[Section 1]{BS}, \cite[Section 5]{BG}, \cite[Section 3]{Vau} or \cite[Chapter 10.25]{Tit}).  However, $f(s,\chi)$ with non-real character has zeros in the half-plane $\Re (s) >1/2$ from (iv) of Theorem \ref{th:gDH1}. Therefore, for any non-real even character $\chi$, the function $f(s,\chi)$ or $g(s,\chi)$ has Riemann's functional equation, real zeros of Dirichlet {\textit{L}}-functions and non-real zeros of Dirichlet series with periodic coefficients not of the form $P(s) L(s,\chi)$ as the title of this paper says. 

In the next subsection, we give some remarks for Riemann's functional equation (\ref{eq:RZfe}). We give proofs of Theorem \ref{th:gDH1} and Corollary \ref{cor:DH2} and some remarks on them in Section 3. 

\subsection{Remarks on Riemann's functional equation}

The following converse theorem on $\zeta (s)$ is widely-known (see also \cite[Chapter 2.13]{Tit}). 
\begin{thm}[Hamburger {\cite[Satz 1]{Ham}}]
Suppose that $F(s)$ satisfies
\begin{equation}
\mbox{$F(s) = \sum_{n=1}^\infty a(n)n^{-s}$, where $a(n) \in {\mathbb{C}}$, converges absolutely for $\sigma >1$}, \tag{H1}
\end{equation}
\begin{equation} 
\mbox{$P(s) F(s)$ is an entire function of finite order for some polynomial $P(s)$}, \tag{H2}
\end{equation}
\begin{equation}
\mbox{$\xi_F (1-s) = \xi_F (s)$, where $\xi_F (s) := \pi^{-s/2} \Gamma (s/2) F(s) \rule{0pt}{2.5ex}$.}
\tag{H3}
\end{equation}
Then, one has $F(s) = C\zeta(s)$, where $C$ is a constant.
\end{thm}
Clearly, the functions $\zeta (s)$ and $f(s,\chi)$ with $\chi (-1)=1$ satisfy the condition (H2) and Riemann's functional equation (H3) from the definitions of them and equalities (\ref{eq:RZfe}) and (\ref{eq:gDH1fe}). However, the functions $L(s,\chi)$ with primitive  characters, $f(s)$ and $V_\pm (s, \chi)$ satisfy (H1) and (H2) but they do not fulfill (H3). It should be mentioned that the functional equation (\ref{eq:gDH1fe}) does not contradict Hamburger's Theorem because $f(s,\chi)$ with $\chi (-1)=1$ does not satisfy (H1). In \cite[Theorem 1]{Kn}, Knopp showed that there are infinitely many linearly independent solutions which satisfy (H2), (H3) and
\begin{equation}
\mbox{$F(s)  = \sum_{n=1}^\infty a(n)n^{-s}$, where $a(n) \in {\mathbb{C}}$, converges absolutely in some half-plane}. \tag{K}
\end{equation}

Knopp gives no explicit representation for the coefficients of $a(n)$ of the Dirichlet series satisfy the condition (K). Using $f(s, \chi)$, we give explicitly some functions satisfy (H1) and (H3) as an analogue or improvement of Knopp's theorem. 
\begin{corollary}\label{cor:DH2}
Put $H(s,q) := (q^s+q^{1-s})^{-1}$ and $g (s,\chi) :=  H (s,q) f (s,\chi)$. Then the function $g(s,\chi)$ with $\chi (-1) =1$ can be expressed as an ordinary Dirichlet series and satisfies Riemann's functional equation 
\begin{equation}\label{eq:gDH1gfe}
g(1-s, \chi) = \Gamma_{\!\! \rm{cos}} (s) g(s, \chi) .
\end{equation}
Namely, the function $g(s,\chi)$ with $\chi (-1) =1$ fulfills (H1), (H3) and 
\begin{equation}
\mbox{$D(s) F(s)$ is an entire function of finite order for some Dirichlet polynomial $D(s)$}. \tag{D}
\end{equation}
\end{corollary}

Let $0< r < q$ be relatively prime integers and put 
\[
Q(s,r/q) := \frac{1}{2 \varphi (q)} \sum_{\chi \!\!\! \mod q} 
\bigl(1+\chi(-1)\bigr) \bigl( \overline{\chi} (r) q^s + G(r, \overline{\chi}) \bigl) L(s,\chi),
\]
where $G(r, \overline{\chi})$ is the generalized Gauss sum defined by $G(r,\overline{\chi}) := \sum_{n=1}^q$ $\overline{\chi}(n)e^{2\pi irn/q}$. In \cite[Corollary 1.5]{NPFE}, it is proved that $H(s,q) Q(s,r/q)$ satisfies (H1), (H3) and (D). It should be emphasised that $g(s,\chi)$ is simpler than $H(s,q) Q(s,r/q)$ because the function $f(s,\chi)$ is expressed as a sum of two Dirichlet $L$-functions.

\section{Proofs}

\subsection{{Proofs of (i) and (iii) of Theorem \ref{th:gDH1}}}

\begin{proof}[Proofs of (\ref{eq:gDH1fe}) and (\ref{eq:gDH2fe})]
Supposes $\chi (-1)= 1$. Then we obtain
\begin{equation}\label{eq:pfgdhf1}
\Gamma_{\!\! \rm{cos}} (s) G(\chi) L(s,\overline{\chi}) = q^{1-s} L(1-s, \chi)
\end{equation}
from (\ref{eq:DLfe}). By replacing $s$ by $1-s$ in (\ref{eq:DLfe}) we have
\[
q^s L(s, \chi) = \Gamma_{\!\! \rm{cos}} (1-s) G(\chi) L(1-s,\overline{\chi}). 
\]
By multiplying both sides  of the equality above by $\Gamma_{\!\! \rm{cos}} (s)$, we obtain
\begin{equation}\label{eq:pfgdhf2}
\Gamma_{\!\! \rm{cos}} (s)  q^s L(s, \chi) = \Gamma_{\!\! \rm{cos}} (s) \Gamma_{\!\! \rm{cos}} (1-s) G(\chi) L(1-s,\overline{\chi}) =
G(\chi) L(1-s,\overline{\chi}) 
\end{equation}
since it holds that
\[
\Gamma_{\!\! \rm{cos}} (s) \Gamma_{\!\! \rm{cos}} (1-s) = 
\frac{2\Gamma (s)}{(2\pi )^s} \cos \Bigl( \frac{\pi s}{2} \Bigr) \cdot \frac{2\Gamma (1-s)}{(2\pi )^{1-s}} \sin \Bigl( \frac{\pi s}{2} \Bigr)
= \frac{\Gamma (s) \Gamma (1-s)}{2\pi} 2 \sin \pi s =1
\]
according to the well-known formulas 
\[
\cos \Bigl( \frac{\pi (1-s)}{2} \Bigr) = \sin \Bigl( \frac{\pi s}{2} \Bigr), \qquad 
\Gamma (s) \Gamma (1-s) = \frac{\pi}{\sin \pi s}. 
\]
Therefore, from (\ref{eq:pfgdhf1}) and (\ref{eq:pfgdhf2}), it holds that
\begin{equation*}
\begin{split}
\Gamma_{\!\! \rm{cos}} (s) f(s, \chi) &= \Gamma_{\!\! \rm{cos}} (s) \bigl( q^s L(s,\chi) + G(\chi) L(s,\overline{\chi}) \bigr) \\
&= G(\chi) L(1-s,\overline{\chi}) + q^{1-s} L(1-s, \chi) =  f(1-s, \chi) 
\end{split}
\end{equation*}
which implies  (\ref{eq:gDH1fe}). Similarly, we can prove (\ref{eq:gDH2fe}), namely, the case $\chi (-1)=-1$.
\end{proof}

\begin{proof}[{Proof of (iii) of Theorem \ref{th:gDH1}}]
Assume that $\chi$ is a real character (${\rm{mod}} \,\, q$). Then we have
\begin{equation}\label{eq:rezepf0}
f (s,\chi) =  \bigl( q^s + i^{-\kappa (\chi)} G(\chi) \bigr) L(s,\chi) . 
\end{equation}
For any real primitive Dirichlet character, all zeros of the factor $q^s + i^{-\kappa (\chi)} G(\chi)$ are on the critical line $\sigma =1/2$ according to the well-known formula
\begin{equation}\label{eq:gaussabs}
\bigl\lvert G(\chi) \bigr\rvert = \sqrt{q} .
\end{equation}
It should be emphasized that this phenomenon is distinctly rare since Dirichlet polynomials $a_0 + a_1 q^{-s}$ have zeros off the vertical line $\sigma =1/2$ for almost all $a_0,a_1 \in {\mathbb{C}}$. Thus, we obtain that all complex zeros of $f(s,\chi)$ with $\Re (s) >0$ are on the vertical line $\sigma =1/2$ if and only if GRH for $L(s,\chi)$ is true. Hence, we have (iii) of Theorem \ref{th:gDH1}.
\end{proof}

\begin{remark}
For $c \in {\mathbb{C}}$, consider the function
\begin{equation}\label{eq:fc}
f_c (s,\chi) := c q^s L(s,\chi) + i^{-\kappa (\chi)} G(\chi) L(s,\overline{\chi})
\end{equation}
as a generalization of $f(s,\chi) = f_1 (s,\chi)$. Clearly, we have 
\[
f_c (s,\chi) =  \bigl( c q^s + i^{-\kappa (\chi)} G(\chi) \bigr) L(s,\chi)
\] when $\chi$ is a real character. By modifying the proof above, we can see that the factor $c q^s + i^{-\kappa (\chi)} G(\chi)$ with $\lvert c \rvert \ne 1$ vanishes for some $\sigma \ne 1/2$ and $t \in {\mathbb{R}}$. Hence, the function $f_c (s,\chi)$ with a real character $\chi$ and $\lvert c \rvert \ne 1$ has complex zeros off the critical line $\sigma =1/2$ even if GRH for $L(s,\chi)$ is true. In this case, we have,
\[
f_c(1-s, \chi) \ne \Gamma_{\!\! \rm{cos}} (s) f_c(s, \chi) \quad \mbox{or} \quad f_c(1-s, \chi) \ne \Gamma_{\!\! \rm{sin}} (s) f_c(s, \chi),
\qquad \lvert c \rvert \ne 1.
\]
Hence, when $\chi$ is real, the functional equations (\ref{eq:gDH1fe}) and (\ref{eq:gDH2fe}) prohibit the existence of complex zeros of $f (s,\chi)$ off the critical line $\sigma =1/2$ under GRH.
\end{remark}

\subsection{{Proofs of (ii) and (iv) of Theorem \ref{th:gDH1}} and Corollary \ref{cor:DH2}}

As usual, we say that two Dirichlet characters $\chi_1,\chi_2$ are non-equivalent if the primitive characters $\chi_1^*,\chi_2^*$ inducing $\chi_1,\chi_2$ are distinct. In \cite[Lemma 8.1]{KaPe}, Kaczorowski and Perelli proved the following. 
\begin{thm}\label{thm:KP1}
For $l= 1, \ldots , m$, let $P_l (s)$ be Dirichlet polynomials and $\chi_l$ be non-equivalent Dirichlet characters such that
\[
\sum_{l=1}^m P_l(s) L(s,\chi_l) = 0 \quad \mbox{identically}.
\]
Then $P_l (s) =0$ identically for $l= 1, \ldots , m$. 
\end{thm}

By using this theorem, we show the statements on non-real zeros of $f(s,\chi)$. 
\begin{proof}[{Proof of (iv) of Theorem \ref{th:gDH1}}]
Let $\chi$ be a non-real Dirichlet character. Then, from Theorem \ref{thm:KP1}, the function $q^{-s} f(s,\chi)$ can not be expressed as $P_0(s) L(s,\chi_0)$, where $P_0(s)$ is a Dirichlet polynomial and $\chi_0$ is a Dirichlet character. Clearly, the Dirichlet series of $L(s,\chi)$ and $q^{-s} G(\chi) L(s,\overline{\chi})$ have periodic coefficients of period $q$ and $q^2$, respectively. Hence, $q^{-s} f(s,\chi)$ is a Dirichlet series with periodic coefficients of period $q^2$ and not of the form $P_0(s) L(s,\chi_0)$. Therefore, by Theorem \ref{thm:SW1}, we have (iv) of Theorem \ref{th:gDH1}.
\end{proof}

We are now in a position to prove the statements on real zeros of $f(s,\chi)$. 

\begin{proof}[{Proof of (ii) of Theorem \ref{th:gDH1}}]
Assume that $L(\sigma_0, \chi) = 0$ for some $\sigma_0 >1/2$. Then, clearly one has $L(\sigma_0, \overline{\chi}) = 0$ and
\[
f(\sigma_0, \chi) = q^s L(\sigma_0, \chi) + i^{-\kappa (\chi)} G(\chi) L(\sigma_0, \overline{\chi}) = 0. 
\]

Conversely, suppose that $L(\sigma, \chi) \ne 0$ for all $\sigma >1/2$. Then, for $\sigma >1/2$, the statement $f(\sigma,\chi) =0$ is equivalent to
\begin{equation}\label{eq:rezepf1}
\frac{L(\sigma, \overline{\chi})}{L(\sigma, \chi)} = - \frac{i^{\kappa (\chi)} q^\sigma}{G(\chi)}.
\end{equation}
The absolute value of the left hand side of (\ref{eq:rezepf1}) is $1$ from 
\begin{equation}\label{eq:absconj1}
\bigl \lvert L(\sigma, \chi) \bigr\rvert = \bigl\lvert \overline{L(\sigma, \chi)} \bigr\rvert = \bigl\lvert L(\sigma, \overline{\chi}) \bigr\rvert .
\end{equation}
On the contrary, the absolute value of the right hand side of (\ref{eq:rezepf1}) is greater than $1$ by (\ref{eq:gaussabs}) if  $\sigma >1/2$. Hence, there are no $\sigma > 1/2$ which satisfies (\ref{eq:rezepf1}). Next suppose that $\sigma =1/2$. Then we have
\[
q^{1/2} L(1/2,\chi) = i^{-\kappa (\chi)} G(\chi) L(1/2,\overline{\chi})
\]
according to the functional equation (\ref{eq:DLfe}). Hence, we obtain
\begin{equation*}
\begin{split}
f(1/2,\chi) &= q^{1/2} L(1/2,\chi) + i^{-\kappa (\chi)} G(\chi) L(1/2,\overline{\chi}) \\ &= 
2  q^{1/2} L(1/2,\chi)  = 2 i^{-\kappa (\chi)} G(\chi) L(1/2,\overline{\chi}) .
\end{split}
\end{equation*}
Therefore, we have $f(\sigma, \chi) \ne 0$ for all $\sigma \ge 1/2$ if $L(\sigma, \chi)$ does not vanish for all $\sigma \ge 1/2$. 

Furthermore, we can easily see that $f(\sigma,\chi) \ne 0$ for all $\sigma \ge 1$ and  primitive Dirichlet character $\chi$ from (\ref{eq:rezepf0}), (\ref{eq:gaussabs}), (\ref{eq:rezepf1}), (\ref{eq:absconj1}) and the fact that $L(\sigma,\chi) \ne 0$ for all $\sigma \ge 1$ which is proved by the existence of the Euler product of $L(s,\chi)$ and the well-known fact $L(1,\chi) \ne 0$ (see \cite[Theorem 10.5.29]{Cohen}). Thus, we obtain (ii) of Theorem \ref{th:gDH1}.
\end{proof}

\begin{remark}
Recall that $f_c (s,\chi)$ is defined by (\ref{eq:fc}) and suppose that $L(\sigma,\chi) \ne 0$ for all $\chi$ and $\sigma \ge 1/2$. By modifying the proof above, for any fixed $\chi_0$ and $\sigma_0 \ge 1/2$, we can find $c_0 \ne 1$ such that $f_{c_0} (\sigma_0,\chi_0) =0$, namely, the function $f_{c_0} (s,\chi_0)$ vanishes on the real line. In this case, we have,
\[
f_c(1-s, \chi) \ne \Gamma_{\!\! \rm{cos}} (s) f_c(s, \chi) \quad \mbox{or} \quad f_c(1-s, \chi) \ne \Gamma_{\!\! \rm{sin}} (s) f_c(s, \chi),
\qquad c \ne 1.
\]
Therefore, there are no real zeros of $f(s,\chi)$ with $\sigma \ge 1/2$ owing to the functional equations (\ref{eq:gDH1fe}) and (\ref{eq:gDH2fe}). However, due to (iv) of Theorem \ref{th:gDH1}, the function $f(s,\chi)$ with a non-real character $\chi$ has non-real zeros in the half-plane $\sigma >1/2$ despite the existence of the functional equations (\ref{eq:gDH1fe}) and (\ref{eq:gDH2fe}). In other words, the functional equations (\ref{eq:gDH1fe}) and (\ref{eq:gDH2fe}) forbid the existence of real zeros of $f (s,\chi)$ but allow the existence of non-real zeros of $f (s,\chi)$ when $\chi$ is non-real. 
\end{remark}

\begin{proof}[{Proof of Corollary \ref{cor:DH2}}]
Clearly, the function $q^{-s} f (s,\chi)$ can be expressed as an ordinary Dirichlet series, namely, satisfies (H1). By the definition of $H(s,q)$,  we have $H(1-s,q) = H(s,q)$ and all poles of $H(s,q)$ are on the vertical line $\sigma =1/2$. Moreover, one has
\[
q^s H(s,q) = \frac{q^s}{q^s+q^{1-s}} = \frac{1}{1+q^{1-2s}} = \sum_{k=0}^\infty \frac{(-q)^k}{q^{2ks}}
\]
if $\sigma > 1/2$. Therefore, the function $g (s,\chi) = q^{-s} f (s,\chi) \cdot q^s H (s,q)$ can be also expressed as an ordinary Dirichlet series. Hence we have Corollary \ref{cor:DH2}. 
\end{proof}

\subsection*{Acknowledgments}
The author was partially supported by JSPS grant 16K05077. 

 

\begin{thebibliography}{1}
\bibitem{Apo} 
T.~M.~Apostol, \textit{Introduction to Analytic Number Theory}. Undergraduate Texts in Mathematics, Springer, New York, 1976.

\bibitem{BS} 
E.~Balanzario and J.~Sanchez-Ortiz, {\it{Zeros of the Davenport-Heilbronn counterexample}}, Math.~Comp. {\bf{76}} (2007), no. 260, 2045--2049. 

\bibitem{BG}
E.~Bomb\'eri and A.~Gosh, {\it{On the Davenport-Heilbronn function}}. translation in Russian Math.~Surveys {\bf{66}}  (2011), no.~2, 221--270. 

\bibitem{Cohen} 
{\rm H.~Cohen}, {\em Number theory. Vol.~II. Analytic and modern tools} (Graduate Texts in Mathematics, 240. Springer, New York, 2007). 

\bibitem{Ham} 
H.~Hamburger, {\it{\"Uber die Riemannsche Funktionalgleichung der $\zeta$-Funktion}}. (German) Math.~Z. {\bf{10}} (1921), no.~3-4, 240--254. 


\bibitem{KaPe}
J.~Kaczorowski and A.~Perelli, {\it{On the structure of the Selberg class. I. $0 \le d \le 1$}}. Acta Math. {\bf{182}} (1999), no.~2, 207--241. 

\bibitem{Kn} 
M.~Knopp, {\it{On Dirichlet series satisfying Riemann's functional equation}}. Invent.~Math. {\bf{117}} (1994), no.~3, 361--372.

\bibitem{NPFE}
T.~Nakamura, {\it{Functional equation and zeros on the critical line of the quadrilateral zeta function}}, to appear in J. Number Theory, arXiv:1910.09837.

\bibitem{Vau}
R.~C.~Vaughan, {\it{Zeros of Dirichlet series}}. Indagationes Mathematicae (New Series) {\bf{26}} (2015), no.~5, 897--909. 

\bibitem{SW}
E.~Saias and A.~Weingartner, {\it{Zeros of Dirichlet series with periodic coefficients}}, Acta Arith. {\bf{140}} (2009), no.~4, 335--344. 


\bibitem{Tit} E.~C.~Titchmarsh, {\it{The theory of the Riemann zeta-function,}} Second edition. Edited and with a preface by D.~R.~Heath-Brown. The Clarendon Press, Oxford University Press, New York, 1986. 

\end{thebibliography}
\end{document}